\documentclass[12pt,reqno]{amsart}

\usepackage{amssymb,amsmath,amsthm,newlfont,enumerate,color,epsfig,graphicx}
\newtheorem{prop}{Proposition}[section]
\newtheorem*{theorem*}{Theorem}
\newtheorem*{definition*}{Definition}
\newtheorem{theorem}[prop]{Theorem}
\newtheorem{proposition}[prop]{Proposition}
\newtheorem{corollary}[prop]{Corollary}
\newtheorem{lemma}[prop]{Lemma}

\newtheorem*{remark}{Remark}

\newcommand{\CC}{{\mathbb C}}
\newcommand{\C}{{\mathbb C}}
\newcommand{\N}{{\mathbb N}}

\newcommand{\cF}{{\mathcal{F}_{\alpha}^2}}
\newcommand{\dF}{{\mathcal{F}_{\alpha}^\infty}}
\newcommand{\cFm}{{\mathcal{F}^2}}
\newcommand{\dFm}{{\mathcal{F}^\infty}}
\newcommand{\supp}{\mathrm{supp\;}}

\newcommand {\dist}{\operatorname{dist}} 
\newcommand{\Hol}{\operatorname{Hol}}

\DeclareMathOperator{\bk}{\textbf{k}}

\title[Multiple sampling and interpolation in Fock spaces]{Geometric conditions for multiple sampling and  interpolation in the Fock space}

\author{A. Borichev, A. Hartmann, K. Kellay, X.  Massaneda }
\address{A.Borichev: Aix-Marseille Univ, CNRS, Centrale Marseille, I2M, 13453 Marseille\\ France}
\email{alexander.borichev@math.cnrs.fr}

\address{A.Hartmann:   Universit\'e de Bordeaux\\
 IMB\\ 351 cours de la Lib\'era\-tion\\ 33405 Talence\\ France}
 \email{Andreas.Hartmann@math.u-bordeaux.fr}  
 
\address{K.Kellay:   Universit\'e de Bordeaux\\
 IMB\\ 351 cours de la Lib\'eration\\ 33405 Talence\\ France}
\email{karim.kellay@math.u-bordeaux.fr}

\address{X.Massaneda: Universitat  de Barcelona\\
Departament de Matem\`a\-tiques i Inform\`atica\\
Gran Via 585, 08007-Bar\-ce\-lo\-na\\ Spain}
\email{xavier.massaneda@ub.edu}



\thanks{First author was supported by Russian Science Foundation grant 14-41-00010. Fourth author was supported by the Generalitat de Catalunya (grant 2014 SGR 289) and the Spanish Ministerio de Ciencia e Innovaci\'on (project MTM2014-51834-P)}

\begin{document}

\begin{abstract}  
We study multiple sampling, interpolation and uni\-queness for the classical Fock spaces in the case of unbounded multiplicities. We show that, both in the hilbertian and the uniform norm case, there are no sequences which are simultaneously sampling and interpolating when the multiplicities tend to infinity. 
This answers partially a question posed by Brekke and Seip.
\end{abstract}

\maketitle

\section{Introduction and main results}

Sampling and interpolating sequences in Fock spaces were characterized by Seip and Seip--Wallst\'en in \cite{S1,S2} by means of a certain Beurling--type asymptotic uniform density.  A consequence of these results is that no sequence can be simultaneously interpolating and sampling, hence there are no unconditional or Riesz bases of reproducing kernels neither.  We refer the reader to the monograph  \cite{S5} by K. Seip for an account on these problems. For a similar result in several variables see \cite{GM}.

We would like to consider sampling and interpolation problems involving a finite, not necessarily bounded, number of derivatives at each point (Hermite type interpolation). The idea of derivative sampling and interpolation is well-known from the theory of band-limited functions. 

Band-limited functions appear naturally in the framework of model spaces and their parent Hardy spaces (we refer to \cite{N2} for model spaces and their importance in operator theory). 
Generalized interpolation, a particular instance of which is derivative interpolation, in Hardy spaces was thoroughly studied by Vasyunin \cite{NV} 
and Nikolski \cite{NNN} 
in the late 1970-s and is characterized 
in terms of a generalized Carleson condition (see \cite[Sections C3.2, C.3.3]{N2} for a comprehensive account). 
%It is also possible to consider such generalized interpolation as well as sampling results in model spaces. Factorisation and commutant lifting are key ingredients available in the Hardy space situation that allow these complete results (the situation gets more intricated in model spaces when the interpolation nodes do not correspond to the zeros of the defining inner function).

%Alternative 1: Apart from those fundamental results in the Hardy space (or model space) situation, considering unbounded sampling and interpolation in other Hilbert spaces of analytic functions when a commutant lifting is not available seems new to us.

%Alternative 2: To our knowledge, since those fundamental achievements, there have been no other significant results for unbounded sampling and interpolation in the setting of Hilbert or more generally Banach spaces (the situation is different for Fr\'echet type spaces).

For the case of sequences with uniformly bounded multiplicity in the classical Fock space a complete description of such derivative sampling and interpolating sequences space was given by Brekke and Seip \cite{BS}. Again, it turns out that there are no sequences which are simultaneously sampling and interpolating in this sense. 

The natural question which then arises, and was already posed by Brekke and Seip, concerns the case when the multiplicities are unbounded. We give conditions on sampling and interpolation in that case. The conditions we obtain are of a somewhat different nature, but still imply that there are no sequences which are both interpolating and sampling, at least when the multiplicities tend to infinity.

We should emphasize that our results make sense when the multiplicities are large. For small multiplicities they essentially give no information.

The results of the following subsection concerning the Hilbert space case were earlier announced in the research note 
\cite{BHKM}. With respect to that research note, observe that the sufficient condition appearing in Theorem~\ref{thm2}(b) has now a weaker and more natural formulation. We will also discuss the situation in the uniform norm in Subsection~\ref{d1.2}. 

\subsection{The Hilbert space case} For $\alpha>0$, define the Fock space  $\cF$ by 
\[
 \cF=\{f\in \Hol(\C): \|f\|^2_2=\|f\|_{\alpha,2}^2:=\frac{\alpha}{\pi}\int_\C|f(z)|^2e^{-\alpha|z|^2}dm(z)<\infty\}.
\]
The constant $\alpha/\pi$ is chosen in such a way that $\|1\|_{\alpha,2}=1$, $dm$ is Lebesgue area measure. 

The Fock space $\cF$ is a Hilbert space with inner product 
$$
\langle f, g\rangle =\frac{\alpha}{\pi}\int_\CC f(z)\overline{g(z)}e^{-\alpha|z|^2}dm(z).
$$
The orthonormalization of the monomials gives the basis 
\[
e_k  (z)=\frac{\sqrt{\alpha^k}}{\sqrt{k!}}z^k,\quad k\ge 0 .
\]
The reproducing kernel  is $k_z(\zeta)=\sum\limits_{k\ge 0} e_k(\zeta) \overline{e_k(z)} =e^{\alpha \bar z \zeta}$, hence
\[
\langle f, k_z\rangle =f(z),\qquad f\in \cF,\quad z\in \C.
\]
It's easy to check that the translations
\[
T_z f(\zeta) =T^\alpha_zf(\zeta):=e^{\alpha \bar z\zeta-\frac{\alpha}{2}|z|^2}f(\zeta-z),\qquad f\in \cF,
\]
act isometrically in $\cF$. 

Let $\bk_z = k_z /\|k_z\|_{2}$ be the normalized reproducing kernel at $z$. 
Note that $T_z 1=\bk_z$. 

A sequence $\Lambda\subset \C$  is called \emph{sampling} for $\cF$ if
\[
 \|f\|_{2}^{2}\asymp \sum_{\lambda\in  
 \Lambda}\frac{|f(\lambda)|^2}{k_\lambda(\lambda)}=\sum_{\lambda\in \Lambda} |\langle f,T_\lambda 1\rangle|^2, 
 \quad f\in\cF,
\]
and \emph{interpolating} if for every $v=(v_\lambda)_{\lambda\in \Lambda}\in \ell^2$, there exists $f\in \cF$ such that
\[
e^{-\frac{\alpha}{2}|\lambda|^2}f(\lambda)=\langle f, T_\lambda 1\rangle =v_\lambda, \qquad \lambda\in\Lambda.
\]
For numerous results on the Fock space and operators acting thereon see the recent book by Zhu \cite{Zhu}.\\

Let us now define sampling and interpolation for the case of higher multiplicity.

We deal with divisors $X$ given as $X=\{(\lambda,m_\lambda)\}_{\lambda\in\Lambda}$, where $\Lambda$ is a sequence of points in $\C$ and $m_\lambda\in\mathbb N$ is the
multiplicity associated with $\lambda$. We associate with each $(\lambda,m_\lambda)$ the subspace 
\[
 N_{\lambda}^2=N_{\lambda,m_\lambda}^2:=\{f\in \cF:f(\lambda)=f'(\lambda)=\cdots=
 f^{(m_\lambda-1)}(\lambda)=0\} .
\]

\begin{definition*}
The divisor $X$ is called \emph{sampling} for $\cF$ if
\[
 \|f\|^2_{2}\asymp \sum_{\lambda\in \Lambda} \|f\|^2_{\cF/N^2_\lambda} =  \sum_{\lambda\in \Lambda}\sum_{k=0}^{m_\lambda-1}|\langle f,T_{\lambda} e_k  \rangle|^2 ,\qquad f\in \cF\ ,
\]
and it is called \emph{interpolating} for $\cF$ if for every sequence 
$$
v=(v_\lambda^{(k)})_{\lambda\in\Lambda,\, 0\le k<m_\lambda}
$$ 
such that
\begin{equation*}
\|v\|_2^2:=\sum_{\lambda\in \Lambda}\sum_{k=0}^{m_\lambda-1}|v_\lambda^{(k)}|^2<\infty
\end{equation*}
there exists a function $f\in \cF$  such that
\begin{equation}\label{interpol}
\langle f,T_{\lambda} e_k  \rangle =v_\lambda^{(k)},\qquad 0\le k <  m_\lambda,\quad \lambda\in \Lambda. 
\end{equation}
\end{definition*}

Equivalently, $X$ is interpolating if for every sequence  $(f_\lambda)_{\lambda\in \Lambda}\subset \cF$ satisfying 
\[
 \sum_{\lambda\in \Lambda} \|f_\lambda\|^2_{\cF/N^2_\lambda} <\infty
\]
 there exists a function $f\in \cF$  such that
\[
 f-f_\lambda\in N^2_\lambda,\qquad \lambda\in \Lambda.
\]

This is exactly the way generalized interpolation is defined in Hardy spaces (see \cite[Section C3.2]{N2} for definitions and results).

An application of the open mapping theorem to the restriction operator 
$$
\mathcal R(f)=\bigl(\langle f,T_\lambda e_k\rangle\bigr)_{\lambda\in\Lambda,\, 0\le k<m_\lambda}
$$
shows that the function $f\in\cF$ such that \eqref{interpol} holds can always be chosen in such a way that $\|f\|_2\le C \|v\|_2$, for some $C>0$ depending only on $X$. The minimal such $C$ is called the \emph{interpolation constant of} $X$, and it will be denoted by $M_X$.

Separation between points in $\Lambda$ plays an important role in our results. Denote by $D(z,r)$ the disc of radius $r$ centered at $z$, $D(r)=D(0,r)$.

\begin{definition*}
A divisor $X$ is said to satisfy \emph{the finite overlap condition} if %there exists $\delta>0$ 
%such that 
\begin{eqnarray}\label{eqsep} 
S_X=
\sup_{z\in \CC}\sum_{\lambda\in \Lambda}\chi_{D(\lambda,\sqrt{m_\lambda/\alpha})}(z)<\infty.
\end{eqnarray}
\end{definition*}

If $\Lambda$ is a finite union of subsets $\Lambda_j$ such that for every $j$, the discs $D(\lambda,\sqrt{m_\lambda/\alpha})$, $\lambda\in\Lambda_j$, are disjoint, 
then $X$ satisfies the finite overlap condition. It is not clear whether the opposite is true.
 
The following two results give conditions for sampling and interpolation in the case of unbounded multiplicities. They are less precise than the results for the bounded case, in that they do not give characterizations. Still, the gap is sufficiently small to show that no divisor can be simultaneously sampling and interpolating when the multiplicities tend to infinity.

\begin{theorem}\label{thm2}
\begin{itemize}
 \item [(a)] If $X$ is a sampling divisor for $\cF$, 
then $X$ satisfies the finite overlap condition and there exists $C>0$  such that 
\[
 \bigcup_{\lambda\in \Lambda} D(\lambda,\sqrt{m_\lambda/\alpha}+C)= \C.
\]

 \item [(b)] Conversely, let the divisor $X$ satisfy 
 the finite overlap condition. There exists $C=C(S_X)>0$ such that if for some compact subset $K$ of $\C$ we have 
\[
 \bigcup_{\lambda\in \Lambda,\, m_\lambda>\alpha C^2} D(\lambda,\sqrt{m_\lambda/\alpha}-C)= \C \setminus K, 
\]
then $X$ is a sampling divisor for $\cF$.
\end{itemize}
\end{theorem}

A separation condition stronger than that in Theorem~\ref{thm2}~(b) permits us to choose a subdivisor with the multiplicities tending to infinity:

\begin{proposition}\label{lemd}
Let the divisor $X$ be such that for every $C>0$ there exists a compact subset $K$ of $\C$ satisfying   
\[
\bigcup_{\lambda\in \Lambda,\, m_\lambda>\alpha C^2} D(\lambda,\sqrt{m_\lambda/\alpha}-C)= \C \setminus K. 
\]
Then 
we can find a subset $\Lambda'$ of $\Lambda$ such that 
for every $C>0$ there is a compact subset $K$ of $\C$ satisfying
\[
\bigcup_{\lambda\in \Lambda',\, m_\lambda>\alpha C^2} D(\lambda,\sqrt{m_\lambda/\alpha}-C)= \C \setminus K, 
\]
and
$$
\lim_{\lambda\in \Lambda',\, |\lambda|\to \infty}m_\lambda=+\infty.
$$
\end{proposition}

\begin{theorem}\label{thm-int}
%Let $X=\{(\lambda,m_\lambda)\}_{\lambda\in\Lambda}$ satisfy the condition $\sup\limits_{\lambda\in\Lambda} m_\lambda=\infty$.
\begin{itemize}
 \item [(a)] If  $X$ is an interpolating divisor for $\cF$, then there exists $C=C(M_X)>0$ such that the discs 
 \newline\noindent $\{D(\lambda,\sqrt{m_\lambda/\alpha}-C)\}_{\lambda\in\Lambda,\,m_\lambda>\alpha C^2}$ are pairwise disjoint.

\item [(b)] Conversely, if for some $C>0$ the discs $\{D(\lambda,\sqrt{m_\lambda/\alpha}+C)\}_{\lambda\in\Lambda}$ are pairwise disjoint, then $X$ is an interpolating divisor for $\cF$.
\end{itemize}
\end{theorem}

\begin{remark} It is easily seen that if $X$ is a divisor such that for some $C>0$ the discs 
$\{D(\lambda,\sqrt{m_\lambda/\alpha}-C)\}_{\lambda\in\Lambda,\,m_\lambda>\alpha C^2}$ are pairwise disjoint, 
and if $\lim\limits_{|\lambda|\to \infty}m_\lambda=+\infty$, then $X$ satisfies the finite overlap condition.
\end{remark}

\begin{corollary}\label{thm0}
Let the divisor $X$ satisfy the condition $\lim_{|\lambda|\to \infty}m_\lambda=+\infty$.
Then $X$ cannot be simultaneously interpolating and sampling for $\cF$.
\end{corollary}

Thus, denoting by $K_\lambda$ the orthogonal complement of $N^2_\lambda$ in $\cF$, we conclude that
the Fock space $\cF$ has no unconditional (Riesz) bases of subspaces $K_\lambda$, if $\lim_{|\lambda|\to \infty}m_\lambda=+\infty$.

\subsection{The $\dF$ Fock space}\label{d1.2}
Let 
 \[
 \dF=\big\{f\in \Hol(\CC)\text{ : } \|f\|_\infty=\|f\|_{\alpha,\infty}:=\sup_{z\in \CC}|f(z)|
e^{-\frac{\alpha}{2}|z|^2}<\infty\big\}.
\]

In order to consider the corresponding $L^\infty$ sampling and interpolation problems, we 
associate to every $\lambda\in \mathbb C$ the subspace 
\[
N^\infty_\lambda=N^\infty_{\lambda,m_\lambda}=\{f\in \dF:f(\lambda)=f'(\lambda)=\cdots=
f^{(m_\lambda-1)}(\lambda)=0\}.
\]

A divisor $X$ is called {\it sampling} for $\dF$,  
if there exists  $L>0$ such that
\[
  \|f\|_{\alpha,\infty}\leq L\sup_{\lambda\in \Lambda} \|f\|_{\dF/N^\infty_\lambda}. 
\]
In a similar way we define the generalized interpolation. The divisor
$X$ is called interpolating for $\dF$ if for
every sequence $(f_\lambda)_\lambda$ with 
\[
  \sup_{\lambda\in \Lambda}\|f_\lambda\|_{\dF/N^{\infty}_\lambda}<\infty, 
\]
there exists a function 
$f\in \dF$ such that
\[
 f-f_\lambda\in N^\infty_\lambda,\quad \lambda\in\Lambda .
\]
As usual, the norm of the interpolating function $f$ is controlled by
the $\ell^\infty$-norm of the sequence $(\|f_\lambda\|_{\dF/N^\infty_\lambda})_{\lambda\in\Lambda}$.
The minimal constant $C$ such that we can always find an interpolating $f$ with $\|f\|_{\alpha,\infty}\le C\,\sup_{\lambda\in\Lambda}\|f_\lambda\|_{\dF/N^\infty_\lambda}$, is called 
the interpolation constant of $X$, and will be denoted by $M_X$. 

\begin{theorem}\label{thm1infty}\begin{itemize}
\item[(a)] If $X$ is a sampling divisor for $\dF$
then there exists $C>0$ such that
$$
 \bigcup_{\lambda\in \Lambda} D(\lambda,\sqrt{m_\lambda/\alpha}+C)=\C.
$$

\item[(b)] Conversely, if for some $C>0$ and a compact set $K$
we have 
$$
 \bigcup_{\lambda\in \Lambda,\; m_\lambda>\alpha C^2} D(\lambda,\sqrt{m_\lambda/\alpha}-C)=\C\setminus K, 
$$
then $X$ is a sampling divisor for $\dF$.
\end{itemize}
\end{theorem}

This result is slightly stronger than Theorem~\ref{thm2} in that for the $p=\infty$ case any $C>0$ is sufficient for sampling while for $p=2$ we need $C>C(S_X)$.

\begin{theorem}\label{thm2infty}
\begin{itemize}
\item[(a)] If $X$ is an interpolating divisor for $\dF$, then there exists $C=C(M_X)>0$ such that
the discs \newline\noindent  $\{D(\lambda,\sqrt{m_\lambda/\alpha}-C)\}_{\lambda\in \Lambda,\, m_\lambda>\alpha C^2}$ are pairwise  disjoint.
\item[(b)] Conversely, there exists $C_0>0$ such that if the discs \newline\noindent$\{D(\lambda,\sqrt{m_\lambda/\alpha}+C_0)\}_{\lambda\in \Lambda}$  are pairwise  disjoint, then $X$ is an interpolating divisor for $\dF$.
\end{itemize}
\end{theorem}

This result is slightly weaker than Theorem~\ref{thm0} in that for the $p=2$ case any $C>0$ is sufficient for interpolation while for $p=\infty$ we need $C\ge C_0$ for some absolute constant $C_0$.

\begin{corollary}\label{thm0inf}
Let the divisor $X$ satisfy the condition $\lim_{|\lambda|\to \infty}m_\lambda=+\infty$.
Then $X$ cannot be simultaneously interpolating and sampling for $\dF$.
\end{corollary}

The proof of this Corollary is completely analogous to that of Corollary~\ref{thm0}.

The problems of sampling and interpolation are linked to that of uniqueness, and thus to zero divisors. 
We say that $X$ is a zero divisor for $\dF$ if $\cap_{\lambda\in \Lambda}N^\infty_{\lambda,m_\lambda}\not=\{0\}$. 
To our knowledge, there is no characterization of the zero divisors of the Fock space. Some conditions are discussed in \cite{L, Zhu}. 

We will establish here a necessary condition for zero divisors which seems sharper than those known so far.

\begin{theorem}\label{thm1}
Let $X$ be a divisor. If there exists a compact subset $K$ of $\C$ such that
\begin{equation}
\label{ud1}
\bigcup_{\lambda\in \Lambda} D(\lambda,\sqrt{m_\lambda/\alpha})=\C\setminus K,
\end{equation}
then $X$ is not a zero divisor for $\dF$. 
\end{theorem}

\subsection{The $\mathcal F^p_\alpha$ Fock spaces}\label{d1.3}
By analogy to the Fock space $\cF$, we can consider the Fock spaces $\mathcal F^p_\alpha$ consisting of  entire functions $f$ such that 
$f(z)\exp(-\alpha|z|^2/2)$ belongs to $L^p(\mathbb C)$, $1<p<\infty$, and the corresponding sampling and interpolation problems. At the end of the paper we discuss how our results extend to $2<p<\infty$ for which we obtain conditions on sampling (Theorem~\ref{thmsamp}) similar to Theorem~\ref{thm2} and on interpolation (Theorem~\ref{thmintp}) similar to Theorem~\ref{thm2infty}. 
% using almost the same arguments as in the proof of these theorems. 
As a consequence, for $2<p<\infty$ we can obtain an analogue of Corollary~\ref{thm0inf}.

It is of interest whether similar results are valid for $1<p<2$.\\

The plan of the paper is the following. In the next section we establish two geometrical results including Proposition~\ref{lemd} and the proof of Corollary~\ref{thm0}. 
In Section~\ref{sec2} we prove the uniqueness Theorem~\ref{thm1}. 
In Section~\ref{sec3} we collect some results on local $L^2$ estimates and the finite overlap condition. Sections~\ref{sec4} and \ref{sec5} are devoted to the proofs of  Theorems~\ref{thm2} and \ref{thm-int}, 
respectively.  
Theorems~\ref{thm1infty} and \ref{thm2infty} are proved in Section~\ref{sec7}. Section~\ref{sec8} is 
devoted to the discussion of the case $2<p<\infty$.\\

Let us comment on the techniques used in this paper. One central result  is that if a function of norm one on a disc has a small quotient norm with respect to functions vanishing up to a given order (adapted to the radius) at the center of the disc, then it will be small in a disc with a slightly smaller radius. In order to obtain such kind of results, we use the maximum principle and some estimates of the incomplete gamma function. This will indeed be the key for the necessity part of the interpolation result and the sufficiency part of the sampling result. Another key ingredient for sampling is a uniqueness result based on the uniform redistribution of the point mass $\Delta(\log|f|)$ over discs centered at the zeros of $f$ with radius adapted to the multiplicity, and measuring the radial growth of the total redistributed mass. The sufficiency of the interpolation part for $p=2$ is based on $\overline{\partial}$-techniques by H\"ormander requiring some subtle choices of the
  weight functions in order to fit to arbitrarily big multiplicities. The case $p=\infty$ adapts a clever trick by Berndtsson allowing to get uniform estimates for the solutions which are optimal in 
the (weighted) $L^2$ norm. 
 
Throughout the remaining part of the paper, we suppose that $\alpha=1$. The general case is dealt with in the same way. 
We define the measure 
$$
d\mu(z)=(1/\pi)e^{-|z|^2}dm(z). 
$$
Denote $\dFm=\mathcal{F}_1^\infty$, $\cFm=\mathcal{F}_1^2$. 
Furthermore, we use the following notations:
\begin{itemize}
\item $A\lesssim B$ means that there is an absolute constant $C$ such that $A \le CB$. 
\item $A\asymp B$  if both $A\lesssim B$ and $B\lesssim A$. 
\end{itemize}

\subsection*{Acknowledgments}
We are thankful to Philippe Charpentier, Daniel Pascuas and %\newline\noindent
St\'ephane Rigat for helpful discussions. 
We are grateful to the referee for interesting comments. The last section of our paper answers one of his questions.

\section{Two geometrical results}
\label{secnew}

%We begin with the following geometrical lemma.

\subsection{An elementary result}

We start with a simple geometric lemma that  will allow us to deduce
Corollary \ref{thm0} from Theorems \ref{thm-int} and \ref{thm2}, and also play a r\^ole in
the proof of Theorem \ref{thm1} in the next section.

\begin{lemma}\label{Lem1}
Suppose that three discs $D_1=D(Z_1,r_1)$, $D_2=D(Z_2,r_2)$, $D_3=D(Z_3,r_3)$ satisfy the property 
$D_1\cap D_2\cap D_3\neq\varnothing$. Then there exist $1\le i,j\le 3$, $i\not=j$ such that
%circles $C_1, C_2$ and $C_3$ meet at a point $z$. Then
\begin{itemize}
%\item[(a)] The opening angle of at least one of the intersections ($\alpha_{ij}$ in
%Figure 1) is at least $\pi/6$; 
\item[(a)] $r_i+r_j-|Z_i-Z_j|\ge c\cdot\min(r_i,r_j)$,
\item[(b)] $m(D_i\cap D_j)\ge c\cdot\min(r_i,r_j)^2$;
\end{itemize}
for some absolute constant $c>0$, with $m(A)$ being the area of $A$.
\end{lemma}

\begin{proof}
The proof is based on elementary geometry. Take $Z\in D_1\cap D_2\cap D_3$. 
We can assume that each pair of discs intersects at exactly two different
points. One of the angles between the vectors $\overrightarrow{ZZ_1}$, $\overrightarrow{ZZ_2}$, $\overrightarrow{ZZ_3}$ is at most $2\pi/3$. 
Without loss of generality we can assume that $\angle Z_1ZZ_2\le 2\pi/3$. Denote by $W_1,W_2$ two points of intersection of $\partial D_1$ and 
$\partial D_2$. Either $\angle W_1Z_1W_2$ or $\angle W_1Z_2W_2$ is at least $\pi/3$. Without loss of generality we can assume that $\angle W_1Z_1W_2\ge \pi/3$.

Now, either $Z_1\in D_2$ or $|W_1-W_2|\ge r_1$. In both situations, (a) and (b) follow immediately.
\end{proof}

%\section{Proof of Corollary~\ref{thm0}}
%\label{sec6}
We are now ready to prove Corollary~\ref{thm0}. For this,
suppose that there is a divisor $X= \{(\lambda,m_\lambda)\}_{\lambda\in\Lambda}$ 
which is simultaneously interpolating and sampling and such that
$\lim_{|\lambda|\to\infty}m_\lambda=\infty$.   By Theorem \ref{thm-int}, there is
a constant $C_1$ such that the discs $D(\lambda,\sqrt{m_\lambda}-C_1)$ are pairwise disjoint so that 
for every $\lambda,\lambda'\in\Lambda$, $\lambda\not=\lambda'$, we have 
\begin{equation}
\sqrt{m_\lambda}+\sqrt{m_{\lambda'}}-|\lambda-\lambda'|\le 2C_1.
\label{xd7}
\end{equation}
Next, by Theorem \ref{thm2}, the discs $D(\lambda, \sqrt{m_{\lambda}+C_1})$ cover the whole plane,
so by Lemma~\ref{Lem1}~(a), we can find pairs $(\lambda,\lambda')$, $\lambda,\lambda'\in\Lambda$, $\lambda\not=\lambda'$ such that 
$$
2C_1+\sqrt{m_\lambda}+\sqrt{m_{\lambda'}}-|\lambda-\lambda'|\ge c\min(\sqrt{m_\lambda},\sqrt{m_{\lambda'}}),\,\, m_\lambda\to\infty,\, m_{\lambda'}\to\infty,
$$
which contradicts \eqref{xd7}. \qed

\subsection{Proof of Proposition \ref{lemd}}
Choose an increasing sequence of positive numbers $(R_n)_{n\ge 1}$, such that 
\begin{equation*}
\C\setminus D(R_n)\subset \bigcup_{m_\lambda>n^2} D(\lambda,\sqrt{m_\lambda}-n).
\qquad\qquad\qquad\qquad (U_n)
\end{equation*}
Iteratively, on the step $s\ge 1$, we remove from $\Lambda$ the subsets 
$$
\Lambda_s=\{\lambda\in\Lambda:|\lambda|>R_s+s,\,(s-1)^2\le m_\lambda<s^2\}.
$$
Then
$$
\lim_{\lambda\in \Lambda',\, |\lambda|\to \infty}m_\lambda=+\infty,
$$
where $\Lambda'=\Lambda\setminus\cup_{s\ge 1}\Lambda_s$. 
%It remains
It remains only to verify that the conditions $(U_n)$ are still valid after each step $s$. Clearly, $(U_n)$, $n\ge s$, do not change.
Now, for $1\le n<s$ we have
\begin{equation}\label{t89}
\C\setminus D(R_n)\subset %\cup_{m_\lambda>n^2} D(\lambda,\sqrt{m_\lambda}-n)\subset 
\Bigl(\bigcup_{m_\lambda>n^2,\,\lambda\notin\Lambda_s} D(\lambda,\sqrt{m_\lambda}-n)\Bigr)\cup
\Bigl(\bigcup_{\lambda\in\Lambda_s} D(\lambda,\sqrt{m_\lambda}-n)\Bigr).
\end{equation}
Observe that if $\lambda\in\Lambda_s$, then $m_\lambda<s^2$ and for every $z\in D(\lambda,\sqrt{m_{\lambda}})$ we have $|z|\ge |\lambda|-\sqrt{m_{\lambda}}>R_s+s-s=R_s$. Hence, 
\begin{equation}\label{t89a}
\bigcup_{\lambda\in\Lambda_s} D(\lambda,\sqrt{m_\lambda}-n)\subset \C\setminus D(R_s).
\end{equation}
By $(U_s)$, we have 
\begin{multline}\label{t89b}
\C\setminus D(R_s)\subset \bigcup_{m_\lambda>s^2} D(\lambda,\sqrt{m_\lambda}-s)\\ \subset \bigcup_{m_{\lambda}>s^2}D(\lambda,\sqrt{m_{\lambda}}-n)\subset\bigcup_{m_\lambda>n^2,\,\lambda\notin\Lambda_s} D(\lambda,\sqrt{m_\lambda}-n).
\end{multline}
Finally, \eqref{t89}--\eqref{t89b} yield 
$$
\C\setminus D(R_n)\subset\bigcup_{m_\lambda>n^2,\,\lambda\notin\Lambda_s} D(\lambda,\sqrt{m_\lambda}-n).\qquad\qquad \qed
$$

\section{Zero divisors}
\label{sec2}

\begin{lemma}\label{zero1}
If $X=\{(\lambda_k,m_k)\}_{k\ge 1}$ is a zero divisor for $\dFm $, then %there exists $A>0$ such that
\[
2\int_{D(R)} \sum_{k\ge 1} \chi_{D(\lambda_k ,\sqrt{m_k})} (z)\log\frac{R}{|z|}dm(z)\le \pi R^2+O(1), \qquad R\to\infty.
\]
\end{lemma}

\begin{proof} %Let $X=\{(\lambda_k,m_k)\}_{k\ge 1}$. 
Assume that $f\in\dFm $ vanishes at $\lambda_k$ of order $m_k$, $k\ge 1$.
Then $\log|f|$ is a subharmonic function, not identically $-\infty$ and bounded above by 
\[
 s(z):=\log\|f\|_\infty+\frac{1}2 |z|^2\ .
\]
In an inductive argument, we will construct a new subharmonic function $h$ such that $\log|f|\le h \le s$ by redistributing the mass of $\Delta (\log|f|)$ of each $D(\lambda_k,\sqrt{m_k})$ uniformly on this disc. %The construction will be  inductively.

Set $h_0:=\log|f|$. Then 
\[
 h_0=m_1\log\Bigl(\frac{|z-\lambda_1|}{\sqrt{m_1}}\Bigr)+u_0,
\]
where $u_0$ is
subharmonic on $\C$ and harmonic in some small neighbourhood of
$\lambda_1$ (since the zeros of the entire function $f$ are isolated).
Define
\[
 v_1(z)=\left\{\begin{array}{ll}
  \displaystyle \frac{ |z-\lambda_1|^2-m_1}{2 } & 
      \text{if }z\in D(\lambda_1,\sqrt{m_1}) \\
      &\\
  \displaystyle m_1\log\Bigl(\frac{|z-\lambda_1|}{\sqrt{m_1}}\Bigl) & \text{otherwise}.
 \end{array}
 \right.
\]
Then $v_1\in\mathcal C^1(\C)$, it is harmonic outside $D(\lambda_1,\sqrt{m_1})$ and it has constant Laplacian $2$ in $D(\lambda_1,\sqrt{m_1})$. Hence the total mass of the measure $\Delta v_1$ on $\C$ is equal to 
$2m(D(\lambda_1,\sqrt{m_1}))=2\pi m_1$, which corresponds to the total mass of the measure 
$\Delta(m_1(\log|z-\lambda_1|/\sqrt{m_1}))=2\pi m_1\delta_{\lambda_1}$.

Now set $h_1=v_1+u_0$, and restart the procedure for $\lambda_2$, i.e. 
write 
\[
 h_1=m_2\log\Bigl(\frac{|z-\lambda_2|}{\sqrt{m_2}}\Bigr)+u_1.
\]
We construct $v_2$ as above and obtain $h_2=v_2+u_1$. 

Iterating this procedure we obtain a sequence of subharmonic functions $(h_n)_n$. We claim that, for every $z\in\C$,
$(h_n(z))_n$ is increasing and 
\[
 \log|f(z)|\le h_n(z)\le s(z).
\]
As above we will give the argument just for the first step. The rest will follow by induction.

Let us begin by showing that $h_1(z)\le s(z)$ for every $z\in\C$. This is clear when $z\notin D(\lambda_1,\sqrt{m_1})$, since for these $z$ we have not changed the function ($h_0=h_1$ outside $D(\lambda_1,\sqrt{m_1})$). This estimate holds also on $\partial D(\lambda_1,\sqrt{m_1})$.

Consider the function $w_1:=v_1+u_0-s$. Then $\Delta (v_1+u_0-|z|^2/2)=\Delta u_0\ge 0$ on $D(\lambda_1,\sqrt{m_1})$, 
since $u_0$ is subharmonic. Hence $w_1$ is subharmonic and it is non-positive on the boundary of $D(\lambda_1,\sqrt{m_1})$. Therefore, it is non-positive throughout this disc, yielding $h_1=v_1+u_0\le s$ also in 
$D(\lambda_1,\sqrt{m_1})$.

The fact that $h_0\le h_1$ is almost obvious. Again, there is nothing to prove for $z$ outside $D(\lambda_1,\sqrt{m_1})$. Inside
the disc it remains to estimate $m_1\log(|z-\lambda_1|/\sqrt{m_1})$ from above by $(|z-\lambda_1|^2-m_1)/2$, 
which is clear, since $r\to m_1 \log (r/\sqrt{m_1})$ is concave, $r\to (r^2-m_1)/2$ is convex, and these functions touch smoothly at $r=m_1$.

The pointwise limit $h$ of the sequence $(h_n)_n$ is still subharmonic (because the sequence is locally eventually stable)  
and by construction it is comprised between $\log|f|$
and $s(z)$. We also know that it is not identically equal to $-\infty$,
since $\log|f|$ is not.

Since $h(z)\le s(z)$, by  Green's formula,
\begin{align*}
 \int_{D(R)}\Delta h(z) \log \frac{R}{|z|}\,dm(z)
 &=-2\pi h(0)+\frac{1}{R}\int_{|z|=R}h(z)\,d|z|\\
 &\le 2\pi \Bigl(-h(0)+\log \|f\|_\infty +\frac{R^2}2 \Bigr).
\end{align*}
On the other hand, 
\[
\int_{D(R)}\Delta h(z) \log \frac{R}{|z|}\,dm(z)\ge 2\int_{D(R)} 
\sum_{k\ge 1} \chi_{D(\lambda_k,\sqrt{m_k})} (z)\log \frac{R}{|z|}\,dm(z),
\]
and the proof is complete.
\end{proof}

\subsection{Proof of Theorem \ref{thm1}}\label{eqAr}
Let $X=\{(\lambda,m_\lambda)\}_{\lambda\in\Lambda}$  be a zero divisor for 
$\dFm $ satisfying \eqref{ud1}. Denote $E=\{z\in\C:\sum_{\lambda \in\Lambda}\chi_{D(\lambda,\sqrt{m_\lambda})}(z)>1\}$. 
Choose $R>0$ such that $K\subset D(R)$.
Then, by Lemma \ref{zero1}, there exists $A>0$ such that
\begin{gather}
\frac{\pi R^2}{2}+A 
\ge \int_{D(R)} \sum_{\lambda \in \Lambda} \chi_{D(\lambda,\sqrt{m_\lambda})} (z)\log\frac{R}{|z|}dm(z)\nonumber\\
=\int_{D(R)} \log \frac{R}{|z|}dm(z)+ \int_{D(R)\cap E} 
 \Big[ \sum_{\lambda \in \Lambda} \chi_{D(\lambda,\sqrt{m_\lambda})} (z) -1\Big] \log \frac{R}{|z|}dm(z) +\nonumber \\
\qquad\qquad\qquad\qquad\qquad + \int_{D(R)\setminus E} \Big[ \sum_{\lambda \in \Lambda} \chi_{D(\lambda,\sqrt{m_\lambda})} (z) -1\Big] \log \frac{R}{|z|}dm(z)\nonumber\\
\ge\int_{D(R)} \log \frac{R}{|z|}dm(z)+ \int_{D(R)\cap E} 
 %\Big[ \sum_{\lambda \in \Lambda} \chi_{D(\lambda,m_\lambda)} (z) -1\Big] 
 \log \frac{R}{|z|}dm(z) %+\nonumber \\ &\qquad\qquad\qquad\qquad\qquad 
 -\int_K \log \frac{R}{|z|}dm(z) \nonumber\\
\ge \frac{\pi R^2}{2}+ \int_{D(R)\cap E} \log \frac{R}{|z|}dm(z)- m(K)\log R -c_1, \label{e7}
 \end{gather}
where $c_1$ depends only on the compact set $K$. % and $mK$ is the area of $K$. 

By Lemma \ref{Lem1}~(b), 
the area of $E$ is infinite. Choose $R_0$ such that $m(E\cap D(R_0))\ge m(K)+1$.
For $R\ge R_0$ we have 
\begin{multline*}
 \int_{D(R)\cap E} \log \frac{R}{|z|}dm(z)\ge  
\int_{D(R_0)\cap E}\log \frac{R}{|z|}dm(z)\\ \ge (m(K)+1)\log\frac{R}{R_0}\ge 
 (m(K)+1)\log R-c_2, 
\end{multline*}
with  $c_2$ depending only on  $R_0$. By \eqref{e7} we get a contradiction, and the proof is completed.
\qed

\section{Local $L^2$ estimates and the finite overlap condition} 
\label{sec3}

Given $x\ge 0$ and $k\in\mathbb Z_+$ we define
\begin{gather*}
\sigma_k(x)=\frac{1}{k!}\int_0^x y^ke^{-y}\,dy,\\
\omega_k(x)=e^{-x}\sum_{s=0}^k\frac{x^s}{s!}.
\end{gather*}
Then $\lim_{x\to+\infty}\sigma_k(x)=\lim_{k\to\infty}\omega_k(x)=1$. Integration by parts gives $\sigma_k(x)+\omega_k(x)=1$.

The following estimates on partial sums of exponentials (or the incomplete Gamma function) will be useful in what follows. 

\begin{lemma}\label{l1}
\begin{itemize}
\item [(a)] Given $t\ge 0$, there exist $\varepsilon>0$ and $k_0>0$ such that for every $k\ge k_0$
\[
\sigma_k(k-t\sqrt{k})\ge \varepsilon.
% e^{-x}\sum_{l=0}^{[x+t\sqrt{x}]}\frac{x^{l}}{l!}\le 1-\varepsilon.
\]
\item [(b)] Given $t\ge 0$, there exists $\varepsilon>0$ such that for every $k\ge 0$
\[
\omega_k(k+t\sqrt{k})\ge \varepsilon.
% e^{-x}\sum_{l=0}^{[x+t\sqrt{x}]}\frac{x^{l}}{l!}\le 1-\varepsilon.
\]
\item [(c)] Given $\varepsilon>0$, there exist $t>0$ such that
for every $k\ge t^2$
\[
\sigma_k(m-t\sqrt{m})\le \varepsilon \sigma_k(m),\qquad t^2\le m\le k.
% e^{-x^2}\sum_{l=0}^{[x-t\sqrt{x}]}\frac{x^{l}}{l!}\le \varepsilon.
\]
%\item [(b)] For every $\varepsilon>0$ there is a $t>0$ and a $x_0>0$ such that
%for all $x>x_0$
%\[
% e^{-x^2}\sum_{l=0}^{[x-t\sqrt{x}]}\frac{x^{l}}{l!}\le \varepsilon.
%\]
\end{itemize}
\end{lemma}

%Observe the quantifiers in both items, and note that in (a) we can choose negative $t$.

\begin{proof} This lemma is essentially contained in \cite{T}; alternatively, one can use the Stirling formula to evaluate the sums $\omega_k(x)$. 
Yet another proof of parts (a) and (b) uses the Poisson law and its approximation by the standard normal law via the central limit theorem. 

To prove (c), we will check that for every $\varepsilon>0$, there exist $t>0$ satisfying 
\begin{equation}
(y-t\sqrt{y})^ke^{-(y-t\sqrt{y})}\le \varepsilon y^ke^{-y},\qquad t^2\le y\le k.
\label{qm}
\end{equation}
Assuming for the moment \eqref{qm}, and integrating from $t^2$ to $x\in[t^2,k]$, we get 
\begin{align*}
\int_{0}^{x-t\sqrt{x}}u^ke^{-u}\,\frac{du}{1-t/(2\sqrt{y})}&=\int_{t^2}^{x}(y-t\sqrt{y})^ke^{-(y-t\sqrt{y})}\,dy\\ &\le \varepsilon \int_{t^2}^{x}y^ke^{-y}\,dy,
\end{align*}
where $u=y-t\sqrt{y}$. Since $y\ge t^2$, we have $1-t/(2\sqrt{y})\ge \frac12$, and hence,
$$
\int_{0}^{x-t\sqrt{x}}u^ke^{-u}\,du\le 2\varepsilon \int_{t^2}^{x}y^ke^{-y}\,dy,
$$
which completes the proof.

To verify \eqref{qm}, we use that 
\begin{align*}
t\sqrt{y}+k\log\bigl(\frac{y-t\sqrt{y}}{y}\bigr)&=t\sqrt{y}-\sum_{j\ge 1}\frac{kt^j}{jy^{j/2}}\\&=-(k-y)\frac{t}{\sqrt{y}}-\frac{kt^2}{2y}-\sum_{j\ge 3}\frac{kt^j}{jy^{j/2}}\le -\frac{t^2}{2}\\
&\le \log \varepsilon
\end{align*}
for $t^2\le y\le k$ and $t\ge \sqrt{2\log(1/\varepsilon)}$.
\end{proof}

\begin{lemma}\label{eqsomme} For every $A\ge 0$ 
there exists $C(A)>0$, $n(A)>0$ such that for every $f\in \cFm$, $n\ge n(A)$, $\lambda\in\C$, we have 
\[
\sum_{k=0}^{n-1}|\langle f,T_{\lambda} e_k  \rangle |^2
\le C(A)\int_{D(\lambda,\sqrt{n}-A)}|f(z)|^2\,d\mu(z).
\]
\end{lemma}

\begin{proof}
Let $g=T_{-\lambda}f$. Our statement is equivalent to 
\begin{equation}\label{pro}
\sum_{k=0}^{n-1}|\langle g, e_k  \rangle |^2\lesssim\int_{D(\sqrt{n}-A)}|g(z)|^2\,d\mu(z),
\end{equation}

Let 
$$
g=\sum_{k\ge 0}a_ke_k.
$$
We use that $e_k$ are mutually orthogonal with respect to $\chi_{D(R)}d\mu$ and 
\begin{equation}\label{norm}
\int_{D(R)}|e_k(z)|^2\,d\mu(z)=\sigma_k(R^2).
\end{equation}
Therefore, \eqref{pro} can be rewritten as 
\begin{equation}\label{dd6}
\sum_{k=0}^{n-1}|a_k|^2\lesssim\sum_{k\ge0} \sigma_k((\sqrt{n}-A)^2)|a_k|^2.
\end{equation}
By Lemma~\ref{l1}~(a), for $n> k\ge k_0(A)$, we have 
$$
\sigma_k((\sqrt{n}-A)^2)\ge \sigma_k(k-2A\sqrt{k})\ge \varepsilon(A)^2>0.
$$
For every $k$ such that $0\le k<k_0(A)$, $n\ge n(A,k)$ we have
$$
\sigma_k((\sqrt{n}-A)^2)\ge   \frac12.
$$
Therefore, for $n\ge \max_{0\le k<k_0(A)}[n(A,k)]$ we obtain 
$$
\sigma_k((\sqrt{n}-A)^2)\ge   \min\Bigl(\frac12,\varepsilon(A)^2\Bigr)>0,\qquad k\ge 0,
$$
and \eqref{dd6} follows.
\end{proof}

\begin{lemma}\label{lem1eq}
Let $X=\{(\lambda,m_\lambda)\}_{\lambda\in\Lambda}$. The divisor $X$ satisfies the finite overlap condition if and only if there exists $C>0$ satisfying 
 \begin{equation}
 \label{1eq}
 \sum_{\lambda\in \Lambda}\sum_{k=0}^{m_\lambda-1}
 |\langle f,T_{\lambda} e_k  \rangle |^2\le C \|f\|^2_{2},\qquad f\in \cFm.
 \end{equation}
\end{lemma}

\begin{proof} 
Suppose that \eqref{1eq} holds. Given $z\in\C$, set $f=T_{z}1$; we have $\|f\|_2=1$. Next, 
\begin{align*}
|\langle f,T_\lambda e_k  \rangle|&=|\langle T_{z}1,T_\lambda e_k  \rangle|=
|\langle  e_0,T_{\lambda -z}e_k  \rangle|= |T_{\lambda -z}e_k  (0)|\\&=
\frac{1}{\sqrt{k!}}|\lambda-z|^ke^{-\frac{1}{2}|\lambda-z|^2}.
\end{align*}
Denote $\rho_\lambda=|\lambda-z|$. %, so that $|\langle f,T_\lambda e_k  \rangle|^2=1/k!e^{-\rho_\lambda^2}$.  If $z\in D(\lambda,\sqrt{m_{\lambda}})$, 
%then $\rho_\lambda^2\le m_{\lambda}$. 
Then \eqref{1eq} implies that   
\begin{align*}
1=\|f\|_{2}^2&\gtrsim \sum_{\lambda\in \Lambda}\sum_{k=0}^{m_{\lambda}-1} \frac{\rho_\lambda^{2k}}{{k!}}e^{-\rho_\lambda^2}\\
&\ge   \sum_{\lambda\in \Lambda,\,\rho_\lambda^2<m_\lambda}  %\chi_{D(\lambda,\sqrt{m_\lambda})}(z)
e^{-\rho_\lambda^2}\sum_{0\le k\le \rho_\lambda^2}\frac{\rho_\lambda^{2k}}{{k!}}\\&=
\sum_{\lambda\in \Lambda}  \chi_{D(\lambda,\sqrt{m_\lambda})}(z)\omega_{[\rho_\lambda^2]}(\rho_\lambda^2), 
%e^{-|\lambda-z|^2}\sum_{0\le k<|\lambda-z|^2}\frac{|\lambda-z|^{2k}}{{k!}}
% \gtrsim  \sum_{\lambda\in \Lambda}  \chi_{D(\lambda,\sqrt{m_\lambda})}(z).
\end{align*}
where $[\rho_\lambda^2]$ is the integer part of $\rho_\lambda^2$. 
By Lemma~\ref{l1}~(b) with $t=0$ we conclude that
$$
1\gtrsim  \sum_{\lambda\in \Lambda}  \chi_{D(\lambda,\sqrt{m_\lambda})}(z).
$$

In the opposite direction, if $X$ satisfies the finite overlap condition, we just apply Lemma~\ref{eqsomme}. 
%,  condition \eqref{eqsep} holds, and 
%\[
%\sum_{\lambda\in \Lambda}\int_{D(\lambda,\sqrt{m_\lambda/\alpha-\delta})}|f(z)|^2e^{-\alpha|z|^2}dm(z)\lesssim\|f\|_2^2
%\]
%Since $D(\lambda,\sqrt{(m_\lambda/\alpha)-\delta}) \subset D(\lambda,\sqrt{m_\lambda/\alpha}-\delta)$ we will be done as soon as we prove the following lemma.
\end{proof}

%%%%%%%%%%%%%%%%%%%%
%%%%%%%%%%%%%%%%%%

\section{Proof of Theorem \ref{thm2}}
\label{sec4}

\subsection{Necessary condition} \label{sampnec2}
Let $X=\{(\lambda,m_\lambda)\}_{\lambda\in \Lambda}$ be a sampling divisor. By Lemma~\ref{lem1eq} it satisfies the finite overlap condition. Suppose that for every $C>0$ 
\[
 \bigcup_{\lambda\in \Lambda} D(\lambda,\sqrt{m_\lambda}+C)\not= \C.
\]
Then there exists a sequence $(z_n)_n\subset\C$ such that
\begin{equation}\label{tz3}
\rho_n:=\dist (z_n,\cup_{\lambda\in \Lambda} D(\lambda,\sqrt{m_\lambda}))   \longrightarrow\infty\ \textrm{as $n\to\infty$}.
\end{equation}
Set $f_n=T_{z_n}1$. %=\bk_{z_n}= k_{z_n} /\|k_{z_n}\|_{2}$ the normalized reproducing kernel. 
%We have $\|f_n\|_2=1$. Since  $X$ is sampling for $\mathcal{F}_\alpha^2$,  b
%By Lemma \ref{lem1eq},   
% $X$ is the union of a finite number of separated varieties,  and b
By Lemma~\ref{eqsomme}, we have 
\begin{gather}
\sum_{%\stackrel{
\lambda\in \Lambda% }{m_\lambda >\alpha\delta^2}
}\sum_{k=0}^{m_\lambda-1} |\langle f_n,T_\lambda e_k  \rangle|^2\lesssim  
 \sum_{\lambda\in \Lambda}\int_{D(\lambda,\sqrt{m_\lambda})} e^{ -|z_n-w|^2}dm(w)\notag\\
 \lesssim\int_{ \cup_{\lambda\in \Lambda} D(\lambda,\sqrt{m_\lambda})}e^{ -|z_n-w|^2}dm(w)
\lesssim \int_{|\zeta|\ge \rho_n} e^{- |\zeta|^2}dm(\zeta)\to 0,\label{tz5}
\end{gather}
and  $X$ cannot be sampling for $\cFm$, which contradicts the hypothesis.

\subsection{Sufficient condition}\label{sampsuff2}

Let us start with a local estimate.

\begin{lemma}\label{Lem3} Given $0<\eta\le 1$ there exists $a(\eta)<\infty$ such that 
if $f\in\cFm $, $m\ge a(\eta)^2$, and if
\begin{gather*}
\sum_{0\le k< m}|\langle f, e_k\rangle |^2\le \eta/2,\\
\int_{D(\sqrt{m})}|f(z)|^2\,d\mu(z)\le 1,
\end{gather*}
then
$$
\int_{D(\sqrt{m}-a(\eta))}|f(z)|^2\,d\mu(z)\le \eta.
$$
\end{lemma}

\begin{proof} Let
$$
f=\sum_{k\ge 0}a_ke_k.
$$
By \eqref{norm}, $(e_k/\sqrt{ \sigma_k(R^2)})_k$ is an orthonormal system with respect to the scalar product $\int_{D(R)} f\overline{g}d\mu$, and it remains
only to verify that 
$$
\sum_{k\ge m}|a_k|^2\sigma_k((\sqrt{m}-a(\eta))^2)\le (\eta/2)\sum_{k\ge m}|a_k|^2\sigma_k(m).
$$
This inequality follows from Lemma~\ref{l1}~(c).
\end{proof}
%\bigskip

Now we are ready to complete the sufficiency part of the proof of Theorem~\ref{thm2}.
Suppose that there exists a sequence $(f_n)_{n\ge 1}$ such that $\|f_n\|_{2}=1$ and 
$$
 \sum_{\lambda\in \Lambda}\sum_{k=0}^{m_\lambda-1}
 |\langle f_n,T_{\lambda} e_k  \rangle |^2\to  0\quad\textrm{as $n\to\infty$}.
$$

Passing to a weakly convergent subsequence denoted again by $(f_n)_{n\ge 1}$ we have two possibilities: either 
(A) $f_n$ converge weakly to $f\not=0$ or (B) $f_n$ converge weakly to $0$. \\

(A):  In this case $X$ is a zero divisor for $f\in\cFm$. Since $\cFm\subset\dFm$, we obtain, by Theorem~ \ref{thm1}, that  
$\mathbb C\setminus \cup_{\lambda\in\Lambda} D(\lambda,\sqrt{m_\lambda})$ cannot be compact, thus contradicting the hypothesis.\\

(B): In this case we define $\eta=(S_X+1)^{-1}$. 
%$$
%A=S_X=\sup_{z\in \C}\sum_{\lambda\in \Lambda}\chi_{D\big(\lambda,\sqrt{m_\lambda}\big)}(z).
%$$
%Take $\eta\in(0,1)$ such that $A\eta<1$. 
We set the constant $C=C(S_X)$ from the formulation of the theorem to be equal to $a(\eta)$ defined in Lemma~\ref{Lem3}. 
Denote $\Lambda_1=\{\lambda\in \Lambda:m_\lambda>a(\eta)^2\}$. 
Then take $B=B(\eta)$ such that 
$$
\C\setminus D(B)\subset \bigcup_{\lambda\in\Lambda_1}D(\lambda,\sqrt{m_\lambda}-a(\eta)) 
$$
and obtain 
\begin{gather*}
1=\int_{\C}|f_n(z)|^2\,d\mu(z)\\ \le \int_{D(B)}|f_n(z)|^2\,d\mu(z)+\sum_{\lambda\in\Lambda_1} \int_{D(\lambda,\sqrt{m_\lambda}-a(\eta))}|f_n(z)|^2\,d\mu(z).
\end{gather*}
Denote by $\Lambda_2$ the set of $\lambda\in\Lambda_1$ such that 
$$
\sum_{k=0}^{m_\lambda-1}
 |\langle f_n,T_{\lambda} e_k  \rangle |^2\le (\eta/2)\int_{D(\lambda,\sqrt{m_\lambda})}|f_n(z)|^2\,d\mu(z).
$$
By Lemma~\ref{Lem3} we obtain that
\begin{gather*}
1\le o(1)+\sum_{\lambda\in\Lambda_2} \int_{D(\lambda,\sqrt{m_\lambda}-a(\eta))}|f_n(z)|^2\,d\mu(z)\\+
\sum_{\lambda\in\Lambda_1\setminus \Lambda_2} \int_{D(\lambda,\sqrt{m_\lambda}-a(\eta))}|f_n(z)|^2\,d\mu(z)\\ \le 
o(1)+\eta\sum_{\lambda\in\Lambda_2} \!\!\int_{D(\lambda,\sqrt{m_\lambda})}|f_n(z)|^2\,d\mu(z)+(2/\eta)
\sum_{\lambda\in \Lambda_1\setminus \Lambda_2}\sum_{k=0}^{m_\lambda-1}|\langle f_n,T_{\lambda} e_k  \rangle |^2\\  \le o(1)+
\eta\sum_{\lambda\in\Lambda} \int_{D(\lambda,\sqrt{m_\lambda})}|f_n(z)|^2\,d\mu(z)\\ \le 
o(1)+\eta S_X\int_{\C}|f_n(z)|^2\,d\mu(z)=o(1)+\eta S_X,\qquad n\to\infty.
\end{gather*}

This contradiction completes the proof. \qed

\section{Proof of Theorem \ref{thm-int}}
\label{sec5}

\subsection{Sufficient condition}\label{intsuff2}
Suppose that the discs 
$D(\lambda,\sqrt{m_\lambda}+C_\Lambda)$, $\lambda\in\Lambda$ are 
pairwise disjoint for some $C_\Lambda>0$.

Let $v=(v_\lambda^{(k)})_{\lambda\in\Lambda,\,0\le k<m_\lambda}$ be a sequence with $\|v\|_2<\infty$ and 
let $(p_\lambda)_{\lambda\in \Lambda}$ be a sequence of polynomials such that
\[
 \langle p_\lambda, e_k\rangle=v_\lambda^{(k)},\qquad \lambda\in\Lambda,\, 0\le k<m_\lambda,
\]
and
\[
 \|p_\lambda\|_{\cFm/N^2_{0,m_\lambda}}=\|p_\lambda\|_\cFm,\qquad \sum_{\lambda\in \Lambda} \|p_\lambda\|^{2}_{\cFm/N^2_{0,m_\lambda}}<\infty.
\]

Set $Q_\lambda=T_\lambda p_\lambda$, $r_\lambda=\sqrt{m_\lambda}+C_\Lambda$, $D_\lambda= D(\lambda,\sqrt{m_\lambda})$, and 
$D'_\lambda= D(\lambda, r_\lambda)$, so that the discs $(D'_\lambda)_{\lambda\in\Lambda}$ are pairwise disjoint. 
Consider first the interpolating function
\[
F(z)=\sum_{\lambda\in \Lambda}Q_\lambda(z)\eta(|z-\lambda|-r_\lambda),
\]
where $\eta$ is a smooth cut-off function on $\mathbb R$ with 
\begin{itemize}
\item $\supp \eta\subset (-\infty,0]$,
\item $\eta\equiv1$ on $(-\infty,-C_\Lambda]$, 
\item $|\eta'|\lesssim 1$.
\end{itemize}

Note that $\supp F\subset \cup_{\lambda\in\Lambda} D'_\lambda$ and 
\[
\int_\CC |F(z)|^2e^{-|z|^2}dm(z)\lesssim 
\sum_\lambda  \int_{D'_\lambda  } |Q_\lambda  (z)|^2e^{-|z|^2}dm(z)\le  \sum_\lambda  \|Q_\lambda  \|_2^2<\infty.
\]
Furthermore, $F$ interpolates $(v_\lambda^{(k)})_{\lambda,k}$ on $X$, that is for every $\lambda\in\Lambda$, the function 
$$
F-\sum_{k=0}^{m_\lambda-1}v_\lambda^{(k)}T_\lambda e_k
$$
vanishes at $\lambda$ of order at least $m_\lambda$.

We search for a holomorphic interpolating function of the form $f=F-u$. This leads to the $\bar\partial$-equation $\overline{\partial}u=\overline{\partial}F$, which we solve using H\"ormander's result \cite[Theorem 4.2.1]{H}: if $\psi$ is a subharmonic function, then there exists $u$ such that $\overline{\partial} F=\overline{\partial} u$ and 
\begin{equation}
\int_\CC |u|^2e^{-\psi}{dm}\le 4\int_\CC |\overline{\partial} F|^2e^{-\psi}\frac{dm}{\Delta\psi}.\label{dx1}
\end{equation}
%(Here and throughout this proof the normalization of the Laplacian is $\Delta=\partial\bar\partial$). 
We set $\psi(z)=|z|^2+v(z)$, where 
$$
v(z)=\sum_{\lambda\in\Lambda}m_\lambda  \Big[\log\frac{|z-\lambda |^2}{m_\lambda}+
1-\frac{|z-\lambda |^2}{m_\lambda}\Big]\chi_{D_\lambda  }(z).
$$%\end{align*}
%Here $T_m(t)=-\sum_{j=1}^{m}{(1-t)^j}/{j} $ indicates the $m$-th Taylor polynomial of $\log t$ around $t_0=1$

Since 
$$
\log(1-t)+t\le 0,\qquad 0\le t\le 1,
$$
we have $v(z)\le 0$. 
Next, $\supp \overline{\partial} F\subset \cup_{\lambda\in\Lambda} (D'_\lambda\setminus D_\lambda)$. %, where 
%\[
%A_\lambda=\bigl\{z\in \CC: \sqrt{m_\lambda}\le |z-\lambda |\le r_\lambda\bigr\}=D'_\lambda\setminus D_\lambda.
%1-\frac1e \le \frac{|z-\lambda |^2}{r_\lambda  ^2}\le 1\Bigr\}.
%\]
Furthermore, for $z\in D'_\lambda\setminus D_\lambda$, $\lambda\in\Lambda$, we have $v(z)=0$ and 
\begin{equation}
\label{qm1}
 |\overline{\partial}F(z)|\lesssim |Q_\lambda  (z)|.
\end{equation}
A direct computation yields 
$$
\Delta v=-4+4\pi m_\lambda\delta_\lambda
$$
on $D_\lambda$, where $\delta_\lambda$ is the unit mass at $\lambda$. 
Therefore,
\begin{equation*}
%\label{xd2}
\Delta \psi=4\pi m_\lambda\delta_\lambda
%\ge 1-\frac{m_\lambda  }{m_\lambda  +C_\Lambda}=\frac{C_\Lambda}{r_\lambda  ^2}\qquad z\in A_\lambda.
\end{equation*}
on $D_\lambda$.
Next, 
\begin{equation}
\label{vvv}
\Delta \psi=4\text{\ on\ }\mathbb C\setminus \bigcup_{\lambda\in\Lambda}D_\lambda. 
\end{equation}

By \eqref{dx1}--\eqref{vvv}, we obtain 
\begin{gather*}
\int_\CC |u|^2e^{-|z|^2}{dm}\le \int_\CC |u|^2e^{-\psi}{dm}\le 4\int_\CC |\overline{\partial} F|^2e^{-\psi}\frac{dm}{\Delta\psi}\\
\lesssim 
\sum_{\lambda\in\Lambda} \int_{D'_\lambda\setminus D_\lambda} |Q_\lambda  (z)|^2e^{-|z|^2}\,dm(z)
\lesssim \sum_\lambda   \|Q_\lambda  \|^{2}_{\cFm/N^2_\lambda}.
\end{gather*}
Thus, $f=F-u\in\cFm$. 

Finally, since in a neighbourhood of $\lambda$ holds the estimate $e^{-\psi(z)}\asymp e^{-v(z)}\asymp |z-\lambda|^{-2m_\lambda}$ and $u$ is analytic,  the bound
\[
 \int_\CC |u(z)|^2 e^{-\psi(z)}{dm(z)}<\infty
\]
forces $u$ to vanish at order $m_\lambda$ at $\lambda$, so that $f$ interpolates the values $(v_\lambda^{(k)})_{\lambda\in\Lambda,\,0\le k<m_\lambda}$, i.e.:
\[
 \langle f, T_\lambda e_k\rangle=\langle Q_\lambda, T_\lambda e_k\rangle=v_\lambda^{(k)},\qquad \lambda\in\Lambda,\, 0\le k<m_\lambda.
\]

\subsection{Necessary condition}\label{intnec2}
Let $X=\{(\lambda,m_\lambda)\}_{\lambda\in\Lambda}$ be an interpolating divisor and assume that %there be no $C>0$ such that 
the discs 
$\{D(\lambda,\sqrt{m_\lambda}-A)\}_{\lambda\in\Lambda,\,m_\lambda>A^2}$ are not separated for some $A>0$ to be chosen later on. Then there exist  
$\lambda,\lambda'\in\Lambda$ and $w\in\C$ such that 
$$
D(w,1)\subset D(\lambda,\sqrt{m_\lambda}-A+1)\cap D(\lambda',\sqrt{m_{\lambda'}}-A+1).
$$
Since $X$ is an interpolating divisor, there exists $f\in\cFm$ such that 
\begin{gather*}
f\in N^2_{\lambda,m_\lambda},\\ 
f-T_{w}1\in N^2_{\lambda',m_{\lambda'}},\\
\|f\|_{2}\le M_X.
\end{gather*}

By Lemma~\ref{Lem3},
\begin{multline*}
\int_{D(\lambda,\sqrt{m_{\lambda}}-A+1)}|f(z)|^2\,d\mu(z)+\int_{D(\lambda',\sqrt{m_{\lambda'}}-A+1)}|(f-T_{w}1)(z)|^2\,d\mu(z)\\=o(1)\cdot M_X^2,\qquad 
A\to\infty,
\end{multline*}
and hence,
\begin{multline*}
\int_{D(w,1)}|f(z)|^2\,d\mu(z)+\int_{D(w,1)}|(f-T_{w}1)(z)|^2\,d\mu(z)\\=o(1)\cdot M_X^2,\qquad 
A\to\infty.
\end{multline*}
On the other hand,
$$
\int_{D(w,1)}|(T_{w}1)(z)|^2\,d\mu(z)
$$
is a positive constant. This gives a contradiction when $A>A(M_X)$. \qed

\section{Proofs of Theorems~\ref{thm1infty} and \ref{thm2infty}}
\label{sec7}

Let us start with some technical lemmas. First of all, given $m\ge 1$, the function 
$\varphi(t)=\varphi_m(t)=t^2/2-m\log t$ decreases on 
$(0,\sqrt{m})$ and increases on $(\sqrt{m},+\infty)$.

\begin{lemma}\label{Lemt1} Given $m\ge 1$, 
\begin{itemize}
\item[(a)] for every $a>0$,  
$$
\varphi(\sqrt{m}+a)\le \varphi(\sqrt{m})+a^2,
$$
\item[(b)] for every $a>0$,   
$$
\varphi(\sqrt{m}-a)\ge \varphi(\sqrt{m})+a^2,\qquad m\ge a^2.
$$
%\item[(c)] for every $A>0$ there exists $a>0$ such that 
%$$
%\varphi(\sqrt{m}-a)\ge \varphi(\sqrt{m})+A,\qquad m\ge a^2,
%$$
\end{itemize}
\end{lemma}

\begin{proof} Since $\varphi'(\sqrt{m})=0$, $\varphi''(t)<2$, $t>\sqrt{m}$, and $\varphi''(t)>2$, 
$0<t<\sqrt{m}$, we have 
$$
\varphi(\sqrt{m}+a)-\varphi(\sqrt{m})=\int_{\sqrt{m}}^{\sqrt{m}+a}\varphi'(t)\,dt< 2\int_0^a s\,ds=a^2
$$
and
$$
\varphi(\sqrt{m}-a)-\varphi(\sqrt{m})=-\int^{\sqrt{m}}_{\sqrt{m}-a}\varphi'(t)\,dt> 2\int_0^a s\,ds=a^2,\quad m\ge a^2.
$$
\end{proof}

\begin{lemma}\label{Lemt2} Given $\delta>0$ there exist $\varepsilon,\eta>0$ such that if 
$f\in\dFm$, $\|f\|_{\infty}\le 1$, $m\ge \delta^2$,  and 
$\|f\|_{\dFm/N^\infty_{0,m}}<\varepsilon$, then 
$$
|f(z)|\le (1-\eta)e^{|z|^2/2},\qquad |z|\le \sqrt{m}-\delta.
$$
\end{lemma}

\begin{proof} Let $h$ be an entire function such
that the function $g:z\to z^mh(z)$ belongs to $\dFm$ and  
$$
 \|f-z^mh\|_{\infty}<\varepsilon.
$$
Then, using that $\|f\|_{\infty}\le 1$, we obtain 
$$
|z^mh(z)|\le |f(z)-z^mh(z)|+|f(z)|\le (1+\varepsilon) e^{|z|^2/2}, \qquad z\in \mathbb C.
$$
Hence,
$$
|h(z)|\le (1+\varepsilon)\exp\varphi(|z|),\quad z\in\C,
$$
where, as above, $\varphi=\varphi_m$.
Set $|z|=\sqrt m$ and use Lemma~\ref{Lemt1}~(b) to get
\[
 |h(z)|\le (1+\varepsilon) e^{\varphi(\sqrt m)-\varphi(\sqrt m-\delta)} e^{\varphi(\sqrt m-\delta)}\le (1+\varepsilon) e^{-\delta^2} e^{\varphi(\sqrt m-\delta)}.
\]
We will have 
\[
 |h(z)|\le (1-\varepsilon-\eta) e^{\varphi(\sqrt m-\delta)}, \qquad |z|=\sqrt m, 
\]
as soon as we choose $\varepsilon, \eta$ small enough so that
\[
 (1+\varepsilon) e^{-\delta^2}\le 1-\varepsilon-\eta.
\]
By the maximum principle, the above estimate holds then for all $|z|\le \sqrt m$. Using that $\varphi(|z|)\ge \varphi(\sqrt m -\delta)$ for $|z|\le \sqrt m-\delta$ we finally see that
$$
|h(z)|\le (1-\varepsilon-\eta)\exp\varphi(|z|),\qquad |z|\le \sqrt{m}-\delta, 
$$
and 
$$
|f(z)|\le |z^mh(z)|+\varepsilon e^{|z|^2/2}\le (1-\eta)e^{|z|^2/2},\qquad |z|\le \sqrt{m}-\delta. 
% |z|^m(1+\varepsilon)\frac{e^{m/2}}{\sqrt{m}^m}+\varepsilon e^{|z|^2/2}.
$$
\end{proof}

\begin{lemma}\label{Lemt2b} Given $\beta>0$, there exist $\varepsilon,\delta>0$ such that if 
$f\in\dFm$, $\|f\|_{\infty}\le 1$, $m\ge \delta^2$, and 
$\|f\|_{\dFm/N^\infty_{0,m}}<\varepsilon$, then 
$$
|f(z)|\le \beta e^{|z|^2/2},\qquad |z|\le \sqrt{m}-\delta.
$$
\end{lemma}

\begin{proof} We argue as in the proof of Lemma~\ref{Lemt2}. %; just use Lemma~\ref{Lemt1}~(c) instead of 
%Lemma~\ref{Lemt1}~(b).
\end{proof}

\begin{lemma}\label{Lemt3}
Let $\delta>0$. %There exists $L=L(C)$ such that i
If $m\ge 1$ and $f$ is a function analytic in
$D(R)$ and continuous up to $\partial D(R)$, $R=\sqrt{m}+\delta$, satisfying $f(0)=f'(0)=\ldots=f^{(m-1)}(0)=0$ and $|f(Re^{it})|\le e^{R^2/2}$, 
$t\in [0,2\pi]$, then 
\[
|f(z)|\le e^{\delta^2}e^{|z|^2/2},\qquad |z|\le R.
\]
\end{lemma}

\begin{proof}
Let $f(z)=z^{m}g(z)$. The function $g$ is holomorphic in $D(R)$ and continuous up to $\partial D(R)$. By the maximum
principle, 
\[
 \log|g(z)|\le \varphi(R), \qquad |z|\le R.
\]
By Lemma~\ref{Lemt1}~(a), and since $\sqrt{m}$ is the point of minimum for $\varphi$, we have 
$$
\log|g(z)|\le \varphi(\sqrt{m})+\delta^2\le \varphi(|z|)+\delta^2, \qquad |z|\le R,
$$
and hence,
$$
|f(z)|\le e^{\delta^2}e^{|z|^2/2}, \qquad |z|\le R.
$$
\end{proof}

\subsection{Proof  of Theorem \ref{thm1infty}. Sufficient  condition}\label{sampsuffinfty}
Suppose
that there exists a sequence $(f_n)_n$ such that $\|f_n\|_{\infty}=1$,
$n\in\N$ and 
\begin{eqnarray*}
\sup_{\lambda\in\Lambda} \|f_n\|_{\dFm/N^\infty_\lambda}\to 0\quad\textrm{as $n\to\infty$}.
\end{eqnarray*}
%Since  $\|f_n\|_{\infty}=1$, there exist $z_n\in \C$ such that 
%\begin{equation}\label{eqsim}
% |f_n(z_n)|e^{-|z_n|^2/2}\to 1,\quad n\to\infty.
%\end{equation}

Passing to a subsequence converging uniformly on compact subsets and denoted again by $(f_n)$ 
we have two possibilities: 
%Two cases may occur: 
either (A) the sequence $(f_n)_n$ converges to $f\not=0$ or (B) the sequence $(f_n)$ converges to $0$. \\

(A):  In this case $X$ is a zero divisor for $\mathcal{F}^{\infty}$. Then, by Theorem~ \ref{thm1}, the set $\mathbb C\setminus \cup_{\lambda\in\Lambda} D(\lambda,\sqrt{m_\lambda})$ cannot be compact, thus contradicting the hypothesis.\\

(B): In this case we recall our assumption that for a compact set $K$ we have 
$$
\Omega=\bigcup_{\lambda\in\Lambda,\,m_\lambda>C^2}D(\lambda,\sqrt{m_\lambda}-C)=\mathbb C\setminus K.
$$
Next, since $(f_n)_n$ tends to $0$, 
$$
|f_n(z)|<\frac 12 e^{|z|^2/2},\qquad z\in K,\, n\ge n_0.
$$
On the other hand, applying Lemma~\ref{Lemt2} to $T_{-\lambda}f_n$ on $D(\sqrt{m_\lambda})$, $\lambda\in\Lambda$, we obtain that 
$$
|f_n(z)|<(1-\eta(C)) e^{|z|^2/2},\qquad z\in \Omega,\, n\ge n_0,
$$
Hence, $\|f_n\|_{\infty}<1$, $n\ge n_0$, and we arrive at a contradiction.

\subsection{Proof  of Theorem \ref{thm1infty}. Necessary condition}\label{sampnecinfty}
The argument is completely analogous to that in the proof of the necessity part of Theorem~\ref{thm2}. Once again, we choose $(z_n)_n$ satisfying \eqref{tz3} and set 
$f_n=T_{z_n}1$,
$$
g_{\lambda,n}=\sum_{k=0}^{m_\lambda-1}\langle f_n,T_\lambda e_k\rangle T_\lambda e_k,\qquad 
\lambda\in\Lambda.
$$
By \eqref{tz5}, 
$$
\sum_{\lambda\in\Lambda}\|g_{\lambda,n}\|^2_{2}\to 0,\qquad n\to\infty,
$$
and since $\|h\|_\infty\le C\|h\|_2$ for $h\in \cFm$, we have 
$$
\max_{\lambda\in\Lambda}\|g_{\lambda,n}\|_{\infty}\to 0,\qquad n\to\infty.
$$
Since $g_{\lambda,n}-f_n\in N^\infty_\lambda$, $\lambda\in\Lambda$, we obtain that $X$ cannot be sampling for $\dFm$ and 
our proof is completed.

\subsection{Proof of Theorem \ref{thm2infty}}

Here, we need an additional construction.

\begin{lemma}
\label{dd4}
Given $q,a\ge 1$, there exists a function $y_{q,a}\in \mathcal C^1(\mathbb C\setminus\{0\})$, $y_{q,a}(z)=y_{q,a}(|z|)$ such that 
$z\mapsto y_{q,a}(z)-2q^2\log|z|$ is $ \mathcal C^2$-smooth at $0$, 
$$
y_{q,a}(z)=|z|^2,\qquad |z|\ge q+a,
$$
and
\begin{equation}
\label{dd5}
\Delta y_{q,a}(z)\ge \frac{4a}{(q+2a-|z|)^2},\qquad |z|< q+a.
\end{equation}
\end{lemma}

\begin{proof} Consider 
$$
\gamma(t)= \frac{a}{(q+2a-t)^2}
$$
and set 
\begin{equation}
\label{dd81}
b=\int_{D(q+a)}\gamma(|z|)\,dm(z)\le \frac{2\pi(q+a)^2}{q+2a}.
\end{equation}
Solve the Dirichlet problem 
$$
\left\{
\begin{aligned}
\Delta g(z)&=4\gamma(|z|)-\frac{4b}{\pi(q+a)^2},\qquad |z|< q+a,\\
g(z)&=0,\qquad |z|= q+a.
\end{aligned}
\right.
$$
Then $g(z)=g(|z|)$,
$$
\int_{D(q+a)}\Delta g(z)\,dm(z)=0,
$$
and by Green's formula, $\nabla g=0$ on $\partial D(q+a)$.

Next, set 
$$
h(z)=q^2\biggl[\log\Bigl|\frac{z}{q+a}\Bigr|^2+1-\Bigl|\frac{z}{q+a}\Bigr|^2\biggr], \qquad |z|\le q+a.
$$
Then $h=\nabla h=0$ on $\partial D(q+a)$, and
$$
\Delta h=-\frac{4q^2}{(q+a)^2}+4\pi q^2\delta_0 
$$ 
on $D(q+a)$.

Finally, we set
$$
y_{q,a}(z)=
\begin{cases}
|z|^2+g(z)+h(z), \qquad |z|\le q+a,\\
|z|^2, \qquad |z|> q+a.
\end{cases}
$$

It is clear that $y_{q,a}\in\mathcal C^1(\mathbb C\setminus\{0\})$ and 
$z\mapsto y_{q,a}(z)-2q^2\log|z|$ is $ \mathcal C^2$-smooth at $0$. 
It remains to verify \eqref{dd5}. In fact, using \eqref{dd81} and the fact that $q, a\ge 1$, we obtain 
\begin{align*}
\frac14\Delta y_{q,a}(z)\ge&  1+\gamma(|z|)-\frac{b}{\pi(q+a)^2}-\frac{q^2}{(q+a)^2}\\ \ge& \gamma(|z|)= \frac{a}{(q+2a-|z|)^2}, \qquad |z|< q+a.
\end{align*}

\end{proof}

\begin{proof}[Proof of Theorem \ref{thm2infty}] We start with the sufficiency part. 
Suppose that the discs 
$D(\lambda,\sqrt{m_\lambda}+6)$, $\lambda\in\Lambda$, are 
pairwise disjoint.
Denote $\Lambda=(\lambda_n)_{n\ge 1}$. 
Let $(\rho_n)_{n\ge 1}$ be a sequence of data with 
$\sup_{n\ge 1}\|\rho_n\|_{\infty}\le 1$. 
We want to find functions $f_N$ such that 
$$
f_N-\rho_n\in N^\infty_{{\lambda_n}},\qquad 1\le n\le N,
$$
and $\|f_N\|_{\infty}\le M$, $N\ge 1$.
A normal family argument will then complete the proof. 

Set $a=3$, 
$m_n=m_{\lambda_n}$, 
$r_n=\sqrt{m_n}+a$, $D_n= D(\lambda_n,r_n)$, 
$D'_n= D(\lambda_n, r_n+a)$, $D''_n= D(\lambda_n, r_n+a+1)$, 

\[
F_N(z)=\sum_{1\le n\le N}\rho_n(z)\tilde\eta(|z-\lambda_n|-r_n),
\]
with $\tilde\eta$ is a smooth cut-off function on $\mathbb R$ with 
\begin{itemize}
\item $\supp \tilde\eta\subset (-\infty,a]$,
\item $\tilde\eta\equiv1$ on $(-\infty,0]$, 
\item $|\tilde\eta'|\lesssim 1$.
\end{itemize}
Then 
$$
\supp \overline{\partial} F_N\subset Q_N=\bigcup_{1\le n\le N}D'_n\setminus D_n.
$$

Next, using Lemma~\ref{dd4}, we define 
$$
\psi_N(z)=|z|^2+\sum_{1\le n\le N}\Bigl[y_{q(n),a}(z-\lambda_n)-|z-\lambda_n|^2\Bigr],
$$
where $q(n)=\sqrt{m_n}$ (by construction, the function $y_{q(n),a}(z-\lambda_n)- |z-\lambda_n|^2$ vanishes outside $D_n$).
Then
$$
\left\{
\begin{aligned}
\Delta \psi_N(z)&=4,\qquad z\in \mathbb C\setminus \cup_{1\le n\le N} D_n,\\
\Delta \psi_N(z)&\ge  \frac{4a}{(\sqrt{m_n}+2a-|z-\lambda_n|)^2},\qquad z\in D_n,\, 1\le n\le N.
\end{aligned}
\right.
$$
Given $\zeta\in Y_N=\mathbb C\setminus \bigcup_{1\le n\le N}D''_n$, 
we define
$$
w_{N,\zeta}(z)=\frac{1}{|z-\zeta|^3+2a},\qquad z\in \mathbb C.
$$
Furthermore, we set
$$
\Omega_{N}=\frac12 \chi_{Q_N}.
$$
We have
$$
\frac{\Delta w_{N,\zeta}(z)}{w_{N,\zeta}(z)}=9\frac{|z-\zeta|(|z-\zeta|^3-2a)}{(|z-\zeta|^3+2a)^2}.
$$
An elementary calculation (consider separately the cases $t\le 8/3$ and $t>8/3$) shows that
$$
9\frac{t(t^3-2a)}{(t^3+2a)^2}\le 2,\qquad t\ge 0.
$$
Next,
$$
9\frac{t(t^3-2a)}{(t^3+2a)^2}\le \frac{9}{t^2},\qquad t\ge 0.
$$

Therefore,
$$
\frac{\Delta w_{N,\zeta}(z)}{w_{N,\zeta}(z)}\le \min\Bigl(2,\frac{4a}{|z-\zeta|^2}\Bigr), 
$$
and we obtain that on the whole plane, 
\begin{equation}
\Delta w_{N,\zeta}\le w_{N,\zeta}(\Delta \psi_N-4\Omega_N).
\label{d145}
\end{equation}
Indeed, for $z\in \mathbb C\setminus \cup_{1\le n\le N} D_n$ we have $\Delta \psi_N-4\Omega_N\ge 2$ and 
for $z\in D_n$, $1\le n\le N$, we have
$$
\frac{4a}{|z-\zeta|^2}\le \frac{4a}{(\sqrt{m_n}+2a-|z-\lambda_n|)^2}
$$
because $|\zeta-\lambda_n|\ge \sqrt{m_n}+2a+1$, $|z-\lambda_n|\le \sqrt{m_n}+2a$. 

Now, we use a remarkable result by Berndtsson \cite[Theorem 4]{B97} (see also \cite{13}): 
the solution $u_N$ of the equation $\overline{\partial} u_N=\overline{\partial} F_N$ minimizing the 
integral 
$$
\int_{\mathbb C}|u_N(z)|^2e^{-\psi_N(z)}\,dm(z)
$$
satisfies the inequality 
$$
\int_{\mathbb C}|u_N(z)|^2e^{-\psi_N(z)}w_{N,\zeta}(z)\,dm(z)%\\ 
\le 
\int_{\mathbb C}|\overline{\partial} F_N(z)|^2e^{-\psi_N(z)}\frac{w_{N,\zeta}(z)}{\Omega_{N,\zeta}(z)}\,dm(z)
$$
under condition \eqref{d145}, for every $\zeta\in Y_N$. 
Since $\supp \overline{\partial} F_N\subset Q_N$ and $\Omega_{N}=\frac12 \chi_{Q_N}$ we have 
$$
\int_{\mathbb C}|u_N(z)|^2e^{-\psi_N(z)}w_{N,\zeta}(z)\,dm(z)\le 2\int_{Q_N}|\overline{\partial} F_N(z)|^2e^{-|z|^2}w_{N,\zeta}(z)\,dm(z).
$$

Hence, we obtain that
\begin{multline*}
\int_{D(\zeta,1)}|u_N(z)|^2e^{-|z|^2}\,dm(z)\\ \le 
2\sum_{1\le n\le N}\int_{D'_n\setminus D_n}|\overline{\partial} F_N(z)|^2e^{-|z|^2}
\frac{dm(z)}{|z-\zeta|^3+2a}\le C
\end{multline*}
for an absolute constant $C$. 
Next, $\overline{\partial} F_N$ vanishes on $D(\zeta,1)$, and, hence, $u_N$ is analytic in $D(\zeta,1)$. 
By the mean value property applied to the function $u_N(z)\exp(|\zeta|^2/2-z\overline{\zeta})$, we have
$$
|u_N(\zeta)|^2e^{-|\zeta|^2}\le C_1 \int_{D(\zeta,1)}|u_N(z)|^2e^{-|z|^2}\,dm(z).
$$
Thus, 
$$
|u_N(\zeta)|^2e^{-|\zeta|^2}\le C_1.
$$
To extend this estimate to $\zeta\in\bigcup_{1\le n\le N}D''_n$, we just apply 
Lemma~\ref{Lemt3} to the functions $u_N+\rho_n-F_N$ analytic in $D''_n$ and use 
that $u_N$ and, hence, $u_N+\rho_n-F_N$ vanish of order $m_n$ at $\lambda_n$,  
$1\le n\le N$.

It remains to set $f_N=F_N-u_N$.
Then 
$$
|f_N(z)|\le Ce^{|z|^2/2},\qquad z\in \mathbb C,
$$
with $C$ independent of $N\ge 1$ and 
$$
f_N-\rho_n\in N^\infty_{{\lambda_n}},\qquad 1\le n\le N.
$$

Finally, for the necessity part,  
the argument is completely analogous to that in the proof of the necessity part of Theorem~\ref{thm-int}. 
Instead of Lemma~\ref{Lem3} we use Lemma~\ref{Lemt2b}.
\end{proof}

\section{The case $\mathcal{F}^p_{\alpha}$, $2<p<\infty$} \label{sec8}

In this section, we formulate and sketch the proofs of the following two results. Let $2<p<\infty$. 

\begin{theorem}\label{thmsamp} 
\begin{itemize} 
 \item [(a)] If $X$ is a sampling divisor for $\mathcal{F}^p_{\alpha}$,  
then $X$ satisfies the finite overlap condition and there exists $C>0$  such that 
\[
 \bigcup_{\lambda\in \Lambda} D(\lambda,\sqrt{m_\lambda/\alpha}+C)= \C.
\]

 \item [(b)] Conversely, let the divisor $X$ satisfy 
 the finite overlap condition. There exists $C=C(S_X)>0$ such that if for some compact subset $K$ of $\C$ we have 
\[
 \bigcup_{\lambda\in \Lambda,\, m_\lambda>\alpha C^2} D(\lambda,\sqrt{m_\lambda/\alpha}-C)= \C \setminus K, 
\]
then $X$ is a sampling divisor for $\mathcal{F}^p_{\alpha}$.
\end{itemize}
\end{theorem}

\begin{theorem}\label{thmintp}   
\begin{itemize}
\item[(a)] If $X$ is an interpolating divisor for $\mathcal{F}^p_{\alpha}$, then there exists $C=C(M_X)>0$ such that
the discs \newline\noindent  $\{D(\lambda,\sqrt{m_\lambda/\alpha}-C)\}_{\lambda\in \Lambda,\, m_\lambda>\alpha C^2}$ are pairwise  disjoint.
\item[(b)] Conversely, there exists $C_0>0$ such that if the discs \newline\noindent$\{D(\lambda,\sqrt{m_\lambda/\alpha}+C_0)\}_{\lambda\in \Lambda}$  are pairwise  disjoint, then $X$ is an interpolating divisor for $\mathcal{F}^p_{\alpha}$.
\end{itemize}
\end{theorem}

The proofs work essentially the same way as for $p=2$ or $p=\infty$.
We start with two auxiliary results replacing Lemmas \ref{Lemt2b} or \ref{Lem3}, and \ref{Lemt3}.
Again we restrict ourselves to $\alpha=1$.

\subsection{Two lemmas}

\begin{lemma}\label{lemA}
Given $0<\beta<1$, there exist $\delta>0$ such that if 
$f\in\mathcal F^p$, $\|f\|_{p}\le 1$, $m\ge \delta^2$, and 
$\|f\|_{\mathcal F^p/N^p_{0,m}}<\beta/2$, then 
$$
\int_{|z|<\sqrt{m}-\delta}|f(z)|^pe^{-p|z|^2/2}\,dm(z)\le \beta.
$$
\end{lemma}

\begin{proof}
Let $h$ be an entire function such $\|f-z^mh\|^p_{p}<\beta/2$. 
Then, using that $\|f\|_{p}\le 1$, we obtain 
$$
\int_{|z|<\sqrt{m}}|z^mh(z)|^pe^{-p|z|^2/2}\,dm(z)\le  2.
$$

Hence, by the mean value theorem applied to the function $\gamma(r)=r
\int_{0}^{2\pi}|h(re^{it})|^pdt$, we get for some $\sqrt{m}-1<R<\sqrt{m}$ 
$$
\int_{0}^{2\pi}|h(Re^{it})|^p\,dt\le 2\frac{e^{p\varphi(R)}}{R},
$$
where $\varphi(s)=\varphi_m(s)=s^2/2-m\log s$. 
Using the subharmonicity of $|h|^p$, and thus that the concentric means are increasing, we deduce that 
\begin{multline*}
\lefteqn{\int_{|z|<\sqrt{m}-\delta}|z^mh(z)|^pe^{-p|z|^2/2}\,dm(z)}\\
=\int_{0}^{\sqrt{m}-\delta} e^{-p\varphi(r)}\int_0^{2\pi}|h(re^{it})|^p\,dt\,rdr
 \le 2\int_{0}^{\sqrt{m}-\delta} e^{p\varphi(R)-p\varphi(r)}\,\frac{rdr}{R}.
\end{multline*}

Using successively Lemma~\ref{Lemt1}~(a) and (b), we obtain that 
\begin{multline*} 
\int_{0}^{\sqrt{m}-\delta} e^{p(\varphi(R)-\varphi(r))}\,\frac{rdr}{R}
 \le c(p)\int_{0}^{\sqrt{m}-\delta} e^{p(\varphi(\sqrt{m})-\varphi(r))} \,dr\\
 \le c(p)\int_{\delta}^\infty e^{-ps^2} \,ds. %\le \beta/2
\end{multline*}
This can be made arbitrarily small for an appropriate choice of $\delta=\delta(p,\beta)$,
which yields the result.
\end{proof}

\begin{lemma}\label{lemB}
 There exists $C>0$ such that 
if $m\ge 1$ and $f$ is a function analytic in
$D(R)$ and continuous up to $\partial D(R)$, $R=\sqrt{m}+1$, satisfying $f(0)=f'(0)=\ldots=f^{(m-1)}(0)=0$, then 
$$
\int_{|z|<\sqrt{m}}|f(z)|^pe^{-p|z|^2/2}\,dm(z)\le C \int_{\sqrt{m}<|z|<R}|f(z)|^pe^{-p|z|^2/2}\,dm(z).
$$
\end{lemma}

\begin{proof}
Set $f(z)=z^{m}g(z)$, where the function $g$ is holomorphic in $D(R)$ and continuous up to $\partial D(R)$. 
Let 
$$
\int_{\sqrt{m}<|z|<R}|f(z)|^pe^{-p|z|^2/2}\,dm(z)=1.
$$
Then, as in the proof of Lemma~\ref{lemA}, for some $\sqrt{m}<R_1<R$ we have 
$$
\int_{0}^{2\pi}|g(R_1e^{it}))|^p\,dt\le \frac{e^{\varphi(R_1)}}{R_1}.
$$
%where $\varphi$ is the same function as in Lemma A.

Arguing as in the proof of Lemma~\ref{lemA}, we get
\[
 \int_{|z|\le \sqrt{m}}|f(z)|^pe^{-p|z|^2/2}\,dm(z)
% \le \int_{0}^{\sqrt{m}} e^{p\varphi(R_1)-p\varphi(r)}\frac{rdr}{R_1}
 \le \int_{0}^{\sqrt{m}} e^{p\varphi(R_1)-p\varphi(r)}\,dr.
\]

By Lemma~\ref{Lemt1}~(a) we have $\varphi(R_1)\le \varphi(\sqrt{m})+1$, and thus with Lemma~\ref{Lemt1}~(b)
we can control the right hand side by a uniform constant.
%\begin{eqnarray*}
%  \int_{|z|\le \sqrt{m}}|f(z)|^pe^{-p|z|^2/2}dm(z)&\le& e^p
%  \int_{0}^{\sqrt{m}} e^{p\varphi(\sqrt{m})-p\varphi(r)}{dr}\\
% &\le& e^p \int_{0}^{\sqrt{m}} e^{-p(\sqrt{m}-r)^2}{dr}\le C
%\end{eqnarray*}
\end{proof}

\subsection{Proof of Theorem~\ref{thmsamp}}

\subsubsection{Necessary condition}

Let us begin with the finite overlap condition. 
We will do this proof in two steps. First we establish a weaker overlap condition (with a smaller radius
$\sqrt{m_\lambda}-C$), and then use a geometric argument to switch to our finite overlap condition.

Fix $\beta=\int_{|z|\le 1} e^{-p|z|^2/2}dm(z)$, and
let $\delta$ be the corresponding para\-meter from Lemma~\ref{lemA}. Suppose 
%the finite overlap condition is not satisfied. Then 
that there exists $z_n$ such that 
\[
M_{X,\delta+1}(z_n):=
 \sum_{\lambda\in\Lambda}\chi_{D(\lambda,\sqrt{m_\lambda}-\delta-1)}(z_n)\ge n.
\]
Let $\lambda$ be a point in $\Lambda$ with 
$|z_n-\lambda|<\sqrt{m_\lambda}-\delta-1$. Set $m=m_\lambda$. By translation we can assume that 
$\lambda=0$ and $|z_n|<\sqrt{m}-\delta-1$.
Observe that
\begin{eqnarray*}
 \int_{|z|\le \sqrt{m}-\delta} |T_{z_n}1(z)|^pe^{-p|z|^2/2}\,dm(z)
 &\ge& \int_{D(z_n,1)}|T_{z_n}1(z)|^pe^{-p|z|^2/2}\,dm(z)\\
 &=&\int_{|z|\le 1} e^{-p|z|^2/2}\,dm(z)=\beta.
\end{eqnarray*}
By contraposition, from Lemma~\ref{lemA} we deduce that $\|T_{z_n}1\|_{\mathcal{F}^p/N_{\lambda}^p}
\ge \beta/2>0$. This being true for every $D(\lambda,\sqrt{m_\lambda}-\delta-1)$ meeting $z_n$
we get $\sum_{\lambda\in\Lambda}\|T_{z_n}1\|_{\mathcal{F}^p/N_{\lambda}^p}\to+\infty$.
This contradiction establishes the following weak overlap condition:
\begin{eqnarray}\label{wfoc}
 \sup_{z\in\C}\sum_{\lambda\in\Lambda}\chi_{D(\lambda,\sqrt{m_\lambda}-\delta-1)}(z)<\infty.
\end{eqnarray}

We have to pass to discs with radius $\sqrt{m_\lambda}$. Suppose the finite overlap condition is not
true but we have \eqref{wfoc}. Let $w_n$ be a sequence such that 
\[
M_{X,0}(w_n):=
 \sum_{\lambda\in\Lambda}\chi_{D(\lambda,\sqrt{m_\lambda})}(w_n)\ge n.
\]
Suppose that there are at least $n/2$ points $\lambda\in\Lambda$ such that the discs $D(\lambda,\sqrt{m_\lambda})$ contain $w_n$ and have  
radii less than $10(\delta+1)$. Then, since the  
quotient norms of $T_{w_n}1$ at $\lambda$ are minorated uniformly for those $\lambda$, the sum of the $p$-th powers of these quotient norms tends to infinity, which is
impossible.

It remains the case when there are at least $n/2$ points $\lambda\in\Lambda$ such that the discs $D(\lambda,\sqrt{m_\lambda})$ contain $w_n$ and have radii at least $10(\delta+1)$.
We can assume that at least $n/40$ of these points $\lambda$ are in an angle $\Gamma$ with vertex at $w_n$ and with 
opening $\pi/10$, say $\Gamma=\{\zeta:|\arg(\zeta-w_n)|<\pi/20\}$. Set $w'_n=w_n+2(\delta+1)$. Then $\lambda\in\Gamma\cap\Lambda$, 
$|\lambda-w_n|<\sqrt{m_\lambda}$ and $\sqrt{m_\lambda}\ge 10(\delta+1)$ imply together that $|\lambda-w'_n|<\sqrt{m_\lambda}-\delta-1$ 
(we use here that $R\ge 10$, $|\theta|<\pi/20$ $\implies$ $|Re^{i\theta}-2|<R-1$).
Thus, $M_{X,\delta+1}(w'_n)\ge n/40$ in contradiction to \eqref{wfoc}.
\\

To get the second necessary condition, we argue as in Subsection~\ref{sampnecinfty} to get
$$
\sum_{\lambda\in\Lambda}\|g_{\lambda,n}\|^2_{2}\to 0,\qquad n\to\infty.
$$
Then, since $p>2$, %$\|h\|_p\le C\|h\|_2$, 
%$l^2\subset l^p$ ($p>2$), we obtain that 
$$
\Bigl(\sum_{\lambda\in\Lambda}\|g_{\lambda,n}\|^p_{p}\Bigr)^{1/p}
 \lesssim
\Bigl(\sum_{\lambda\in\Lambda}\|g_{\lambda,n}\|^2_{2}\Bigr)^{1/2}\to 0,\qquad n\to\infty.
$$
%Here we use that $p>2$.

\subsubsection{Sufficient condition}

Suppose $\|f\|_p=1$, $\sum_{\lambda}\|f_n\|^p_{\mathcal{F}^p/N_{\lambda}^p}\to 0$, $n\to+\infty$.
Passing to a subsequence we can assume that $f_n$ converges weakly to $f$. If $f\neq 0$ then
$\Lambda$ is a zero divisor in $\mathcal{F}^p\subset \mathcal{F}^{\infty}$, and we get a  contradiction.
%(Here we probably need to pay a little bit attention since the kernels are not in $\mathcal{F}%^q$
%so that we can not directly deduce that $f$ vanishes on $\Lambda$ with the right multiplicity.)
%{\color{black} 
%Here we use that $p>2$.
%}

Otherwise $f_n$ goes weakly to 0. In order to repeat the reasoning of Subsection~\ref{sampsuff2} (with appropriate change to the quotient norm), we
use again Lemma~\ref{lemA} (which replaces here Lemma~\ref{Lem3}).%B (another analog of Lemma 7.3):

%Lemma B. Given $\eta>0$, there exist $\delta>0$ such that if 
%$f\in\mathcal F^p$, $m\ge \delta^2$, then 
%$$
%\int_{|z|<\sqrt{m}-\delta}|f(z)|^pe^{-|z|^2}dm(z)\le \eta \int_{|z|<\sqrt{m}}|f(z)|^pe^{-|z|^2}dm(z).
%$$

\subsection{Proof of Theorem~\ref{thmintp}}

\subsubsection{Necessity}

One can argue like in Subsection~\ref{intnec2} %of our paper 
using the above Lemma~\ref{lemA} instead of Lemma~\ref{Lem3}.  %5.1 used in the paper.

\subsubsection{Sufficiency}

This uses the case $p=+\infty$. Starting from data $(f_{\lambda})\in \ell^p(\mathcal{F}^p/
N_{\lambda}^p)\subset \ell^{\infty}(\mathcal{F}^{\infty}/
N_{\lambda}^{\infty})$ we construct the smooth interpolating function $F$ which is in the right
weighted $L^p$ space. We then solve the $\bar\partial$-problem as for $p=\infty$ and
reach the inequality 
\begin{eqnarray*}
|u_N(\zeta)|^2e^{-|\zeta|^2}&\le&\int_{D(\zeta,1)}|u_N(z)|^2e^{-|z|^2}\,dm(z) \\
 &\le&
2\sum_{1\le n\le N}\int_{D'_n\setminus D_n}|\overline{\partial} F_N(z)|^2e^{-|z|^2}
\frac{dm(z)}{|z-\zeta|^3+2a}.%\le C
\end{eqnarray*}
Since $\overline{\partial} F_N\in L^p(e^{-p|z|^2/2})$, and $p>2$, we observe 
$|\overline{\partial} F_N(z)|^2e^{-|z|^2}\in L^{p/2}(dm)$. Hence, the right hand side is a convolution
of an $L^{p/2}$ function with an $L^1$ function, which, by Young's
inequality, is in $L^{p/2}(dm)$. Thus, the
left hand side is an $L^{p/2}$-function in $\zeta$ (uniformly controlled in $N$). %  - to be checked).
%Hence
%\begin{eqnarray*}
% \lefteqn{\int_{\C} \left(|u_N(\zeta)|^2e^{-|\zeta|^2}\right)^{p/2}dm(\zeta)}\\
% &&\le
% \int_{\C}\left(\int_{D(\zeta,1)}|u_N(z)|^2e^{-|z|^2}\,dm(z) \right)^{p/2}dm(\zeta)\le C
%\end{eqnarray*}

It remains to extend the estimate obtained here to the union of some discs as at the end of the proof of Theorem~\ref{thm2infty}. 
Instead of Lemma~\ref{Lemt3}, we  use here Lemma~\ref{lemB}.

\end{document}